\documentclass[twocolumn]{autart}    
\usepackage{amsmath}
\usepackage{amssymb}
\usepackage[usenames, dvipsnames]{color}
\usepackage{psfrag}

\newtheorem{theorem}{Theorem}

\newtheorem{lemma}[theorem]{Lemma}
\newtheorem{corollary}[theorem]{Corollary}
\newtheorem{assumption}{Assumption}
\newtheorem{remark}{Remark}

\newtheorem{definition}{Definition}

\newcommand{\tr}[1]{#1^{\sf T}} 

\DeclareMathOperator*{\argmin}{arg\,min}
\DeclareMathOperator*{\argmax}{arg\,max}
\DeclareMathOperator{\bd}{bd}

\usepackage{graphicx}       
\definecolor{dblue}{rgb}{ 0.00, 0.00, 0.60 }

\begin{document}

\begin{frontmatter}

\title{Robust MPC via Min-Max Differential Inequalities\thanksref{footnoteinfo}}

\thanks[footnoteinfo]{This paper was not presented at any IFAC meeting. Corresponding author: B. Chachuat}
\author[ICL]{Mario E. Villanueva}\ead{mario.villanueva10@ic.ac.uk},
\author[IMTEK,OPTEC]{Rien Quirynen}\ead{rien.quirynen@esat.kuleuven.be},
\author[IMTEK]{Moritz Diehl}\ead{moritz.diehl@imtek.uni-freiburg.de},
\author[ICL]{Beno\^it Chachuat}\ead{b.chachuat@ic.ac.uk},
\author[STECH]{Boris Houska}\ead{borish@shanghaitech.edu.cn},
\address[ICL]{Centre for Process Systems Engineering, Department of Chemical Engineering, Imperial College London, South Kensington Campus,London SW7 2AZ, UK. }
\address[IMTEK]{Department of Microsystems Engineering (IMTEK) and Department of Mathematics, 
University of Freiburg, Georges-Koehler-Allee 102, 79110 Freiburg, Germany.}
\address[OPTEC]{Electrical Engineering Department, KU Leuven, Kasteelpark Arenberg 10, 3001 Leuven, Belgium.}
\address[STECH]{School of Information Science and Technology, ShanghaiTech University, 319 Yueyang Road, Shanghai 200031, China.}
\date{DRAFT- ACCEPTED}

\begin{abstract}                         
This paper is concerned with tube-based model predictive control (MPC) for both linear and nonlinear, input-affine continuous-time dynamic systems that are affected by time-varying disturbances. We derive a min-max differential inequality describing the support function of positive robust forward invariant tubes, which can be used to construct a variety of tube-based model predictive controllers. These constructions are conservative, but computationally tractable and their complexity scales linearly with the length of the prediction horizon. In contrast to many existing tube-based MPC implementations, the proposed framework does not involve discretizing the control policy and, therefore, the conservatism of the predicted tube depends solely on the accuracy of the set parameterization. The proposed approach is then used to construct a robust MPC scheme based on tubes with ellipsoidal cross-sections. This ellipsoidal MPC scheme is based on solving an optimal control problem under linear matrix inequality constraints. We illustrate these results with the numerical case study of a spring-mass-damper system. 
\end{abstract}

\end{frontmatter}

\section{Introduction}
\label{sec::intro}

Model predictive control (MPC) refers to a class of feedback controllers, which proceed by solving, at each time step, an optimal control problem predicting the future behavior of a dynamic system on a finite, receding time-horizon, using the current state estimate as initial condition~\cite{Rawlings2009b}. The predicted optimal control trajectory is applied to the actual system until the next measurement becomes available, and the process is then repeated. The implementation of such controllers is based on a certainty-equivalence principle, whereby the future of the system is optimized as if neither external disturbances nor model mismatch were present, despite the fact that such disturbances and mismatch are the reason why feedback is needed in the first place.

The main advantage of certainty-equivalence in MPC is that the resulting optimization problems can often be solved efficiently, in real time~\cite{Diehl2009,Houska2011}. This approach works well in many practical applications, and it often exhibits a certain robustness due to its inherent ability to reject disturbances~\cite{Pannocchia2011,Yu2014}. However, the constraints may become violated when large disturbances occur, since uncertainty is not taken into account in optimizing the predicted state trajectories. In such cases, robust MPC schemes can be used to mitigate these optimistic, certainty-equivalence-based predictions~\cite{Rawlings2009b}. Nonetheless, a rigorous formulation of robust MPC calls for the solution, at each sampling time, of an optimization problem whose decision variables are the future control policies, that is, functions mapping the state measurements onto the control actions. Such optimization problems are hard to solve in general, and brute-force approximations, e.g. based on scenario trees~\cite{Dadhe2008,Engell2009}, can currently only be used for very short time-horizons. Because scenario-tree approaches scale exponentially with the length of the time-horizon, they may even be worse than robust dynamic programming approaches~\cite{Bertsekas2007,Diehl2004a,Rawlings2009b}, which scale linearly with the length of the prediction horizon, yet exponentially with the state dimension.

Convex formulations of robust MPC have been derived for certain classes of problems, for instance when the dynamic system is jointly affine in the state, control and uncertainty and the feedback control law is itself affine in the disturbance~\cite{Goulart2006,Goulart2007}. There, the number of the (matrix-valued) optimization variables scales quadratically with the length of the prediction horizon. The conservatism introduced by an affine parameterization of the control law is discussed in~\cite{VanParys2013}. In this context, we also refer to~\cite{Zeilinger2014}, where real-time variants of robust MPC based on certain affine feedback laws are analyzed. Other convex formulations can be obtained by reformulating the semi-infinite constraints arising in robust MPC as linear matrix inequalities (LMIs). One such LMI reformulation for bounding the worst-case performance of linear systems under additive bounded uncertainty using constant state-feedback control laws was derived in \cite{Kothare1996}. Another approach was presented in~\cite{Li2010}, where the future model variations are bounded by a family of polytopes expressed as LMI constraints.

Other state-of-the-art approaches in robust MPC adopt a set-theoretic perspective. These methods find their origins in viability theory~\cite{Aubin1991,Kurzhanski1993,Kurzhanski1997} or, more specifically, in set-theoretic methods for control~\cite{Blanchini1999,Blanchini2008}. Robust MPC schemes based on these parametric set-propagation methods are also known collectively under the name tube-based MPC. There, the predicted trajectory is replaced by a robust forward invariant tube (RFIT) in the state-space, namely a tube that encloses all possible state trajectories under a given feedback control law, which is independent of the uncertainty realization~\cite{Langson2004}. Tube-based approaches are typically analyzed under the assumption that exact state measurements are available~\cite{Rakovic2005}, or that the equations of a parameterized state estimator, e.g. a linear filter, can be added to the system dynamics so that standard tube-based methods transfer readily~\cite{Mayne2009}.

A parameterized tube-based MPC formulation for linear discrete-time systems with affine uncertainty has been proposed in~\cite{Rakovic2012a}. This formulation allows for the simultaneous optimization of tubes and control laws that are nonlinear in the state measurements, resulting in a computationally tractable, linear programming (LP) formulation, whose decision variables and constraints scale quadratically with the prediction horizon. A generalization handling more general cost functions is considered in~\cite{Rakovic2012b}, and a way of reducing the online complexity of this approach to linear complexity via offline computations is further presented in~\cite{Rakovic2012c}. Tube-based methods have also been developed for linear systems with multiplicative uncertainty, for example by using polytopic tubes with quadratic cost, which leads to a quadratic programming (QP) formulation~\cite{Evans2012}. Regarding nonlinear dynamics, a possible tube-based approach involves linearizing the system around a feasible, but suboptimal, trajectory and computing the tube by regarding the linearization errors as additional uncertainty. This idea was used in~\cite{Lee2002} with polytopic tubes and affine feedback laws. A similar approach was developed by~\cite{Cannon2011} in the case of quadratic cost terms and ellipsoidal tubes. A tube-based approach for nonlinear continuous-time systems was proposed in~\cite{Yu2013}, where the feedback control laws are affinely parameterized and computed offline.

This paper presents a novel numerical approach for addressing tube-based MPC problems. In contrast to existing methods which parameterize the control law, our approach introduces a min-max differential inequality exploiting the properties on the boundary of RFITs. These min-max differential inequalities yield a non-trivial generalization of differential inequalities~\cite{Lakshmikantham1969,Walter1970,Scott2013,Villanueva2014} and provide sufficient conditions for a time-varying convex-set-valued function to be a RFIT for a class of continuous-time nonlinear control systems. We show that these (on the first view) rather abstract concepts can be used to derive practical implementations of tube-based MPC, which i) scale linearly with the length of the prediction horizon, and ii) do not rely on a particular parameterization of the control law. In principle, this approach can achieve arbitrary precision, insofar as the tubes are represented with sufficient accuracy. 

The rest of the paper is organized as follows. The problem formulation is described in Sect.~\ref{sec::RPMC}. The main theoretical framework for characterizing RFITs for nonlinear input-affine systems is developed in Sect.~\ref{sec::minmaxDI} and its application to RFITs with ellipsoidal cross-sections is presented in Sect.~\ref{sec::E-RFITs}. A practical implementation of tube-based MPC based on these results is discussed in Sect.~\ref{sec::minmaxMPC} and illustrated with a numerical case study in Sect.~\ref{sec::numerics}. Finally, Sect.~\ref{sec::conclusion} concludes the paper.

\paragraph*{Notation and preliminaries}
The sets of real and positive real numbers are denoted by $\mathbb{R}$ and $\mathbb{R}_{++}$. The sets of compact and compact convex subsets of $\mathbb R^{n}$ are denoted by $\mathbb K^{n}$ and $\mathbb K^{n}_{C}$, respectively. The support function $V[Z]:\mathbb R^{n} \to \mathbb R$ of a set $Z \in \mathbb K^{n}_{C}$ is defined as
\begin{equation*}
\forall c\in \mathbb R^{n},\quad V[Z](c):=\max_{z} \{ \tr{c} z | z \in Z \}\,.
\end{equation*}
Moreover, $\bd Z$ denotes the boundary of $Z$ and $\Pi(Z)$ its power set. The Hausdorff distance between $W,Z \in \mathbb K^n_C$ is given by
\begin{equation}
\label{eq::hausdorffdistdef}
\begin{aligned}
d_\mathrm{H}(W,Z) := \max &\left \{\;\max_{w \in W} \min_{z \in Z}  \left\Vert w - z \right\Vert_2 \right. \,,\\
&\phantom{\big\{\;}\left.\max_{z \in Z} \min_{w \in W} \left\Vert w - z \right\Vert_2 \,\phantom{,} \right\} \,.
\end{aligned}
\end{equation}
A set $Z\in\mathbb K^{n}_{C}$ is said to be strictly convex if each of its supporting hyperplanes meets $\bd Z$ at exactly one point $z\in \bd Z$, and it is called a smooth set if $\bd Z$ is itself a smooth submanifold of $\mathbb{R}^{n}$. Moreover, there exists a $C^\infty$-smooth convex function $g:\mathbb{R}^{n}\to \mathbb{R}$ such that
\begin{align}
\bd Z \: := \: \{ z\in\mathbb{R}^{n} \mid g(z) = 0 \},
\label{eq:bdZ_g}
\end{align}
as discussed in~\cite{Azagra2002}.

Let $\mathcal S^{n-1}$ denote the unit sphere in $\mathbb{R}^{n}$.
Given a smooth set $Z\in\mathbb{K}^n_C$, the Gauss map $\mathcal{G}_{Z}: \bd Z \to \mathcal S^{n-1}$ is a continuous function assigning to every boundary point $z\in\bd Z$ its unique outward normal. It is defined as 
\begin{equation*}
\forall \zeta\in \bd Z, \quad \mathcal{G}_{Z}(\zeta) := \left\Vert\frac{\partial g}{\partial \zeta}(\zeta)\right\Vert^{-1} \frac{\partial g}{\partial \zeta}(\zeta) \;,
\end{equation*}
for any continuously-differentiable and convex function $g$ satisfying \eqref{eq:bdZ_g}. The differential $\partial \mathcal{G}_{Z}/\partial \zeta(\zeta)$ defines a linear operator from $T_{\zeta}Z$, the tangent space of $\bd Z$ at $\zeta$, onto itself. The set $Z$ is said to have positive curvature at $\zeta\in \bd Z$ if
\begin{equation*}
\forall w\in T_{\zeta}Z\setminus\{0\}, \quad \tr{w}\,\frac{\partial \mathcal{G}_{Z}}{\partial \zeta}(\zeta)\,w > 0\,.
\end{equation*}
In particular, any smooth set $Z$ with positive curvature everywhere is also strictly convex. Moreover, if $Z$ is both smooth and strictly convex, then $\mathcal{G}_{Z}$ has a continuous inverse $\mathcal{G}^{-1}_{Z}$, called the inverse Gauss map. In other words, $\bd Z$ is homeomorphic to $\mathcal S^{n-1}$ through $\mathcal{G}_{Z}$.

The set of $n$-dimensional Lebesgue-integrable functions on the interval $I\subseteq\mathbb R$ is denoted by $\mathbb{L}(I)^{n}$, or simply $\mathbb{L}^n$ if $I=\mathbb{R}$. Unless otherwise stated, Lebesgue integration is understood with respect to the time variable. The abbreviation $\operatorname{a.e.}$ is used to indicate that a property holds almost everywhere.

The sets of $n\times n$ symmetric positive semi-definite and symmetric positive definite matrices are denoted by $\mathbb S^{n}_{+}$ and $\mathbb S^{n}_{++}$, respectively. Ellipsoids in $\mathbb R^{n}$ with center $q\in\mathbb R^{n}$ and positive semi-definite shape matrix $Q\in\mathbb S_+^{n}$ are defined as
\begin{equation*}
\mathcal E(q,Q) := \left\{ q + Q^\frac{1}{2} v \mid \tr{v}v \leq 1  \right\} \,,
\end{equation*}
with $Q^\frac{1}{2}$ being the symmetric square-root of $Q$. By a small abuse of notation, $\mathcal{E}(Q)$ denotes the ellipsoid with shape matrix $Q$ and centered at zero. The (Moore-Penrose) pseudoinverse of a matrix $A\in\mathbb{R}^{m\times n}$ is denoted by $A^{\dagger}$, and its Frobenius norm by $\|A\|_F := \sqrt{\operatorname{Tr}(\tr{A}A)}$.

\section{Problem Formulation}
\label{sec::RPMC}

Consider a nonlinear control system in the form:
\begin{align}
\label{eq::ODE}
\dot x(t) & = \ f(x(t),w(t)) + G(x(t))u(t)\\ 
&=: \: g(x(t),u(t),w(t)) \, ,\nonumber
\end{align}
where $f: \mathbb R^{n_x}\times \mathbb R^{n_w}\to\mathbb R^{n_x}$,
$G:\mathbb R^{n_x}\to\mathbb R^{n_x\times n_u}$ and $g:\mathbb R^{n_x}\times\mathbb R^{n_u}\times\mathbb R^{n_w}\to\mathbb R^{n_x}$ are potentially nonlinear functions, for which regularity assumptions will be stated later on, as necessary; the state trajectory is denoted by $x \in \mathbb{L}^{n_x}$; $u \in \mathbb U := \left\{ u \in \mathbb{L}^{n_u} \mid \forall t \in \mathbb R, \; u(t) \in U \subseteq \mathbb R^{n_u} \right\}$ denotes the
control; and $w\in\mathbb W := \{ w \in \mathbb{L}^{n_w} \mid \forall t \in \mathbb R, \; w(t) \in W \subseteq\mathbb R^{n_w}\}$ denotes the exogenous disturbance. 

The class of nonlinear control systems~\eqref{eq::ODE} is affine in the control function $u$. Although the reasons for this assumption will become apparent later on, it is important to note that it is not as restrictive as it may seem. In engineering practice, many physical systems possess such an affine control structure. Moreover, any nonlinear controlled system may be reformulated into the desired form under a stronger assumption on $u$~\cite{Houska2014}; for instance, under the assumption that $u$ is at least locally Lipschitz continuous, an integrable control $v$ can be introduced such that $u$ is now regarded as a auxiliary state satisfying the differential equation $\dot{u}(t) = v(t)$.

\begin{assumption}
\label{ass::compact}
The sets $U \subseteq \mathbb R^{n_u}$ and $W \subseteq \mathbb R^{n_w}$ are compact and convex, i.e. $U\in\mathbb{K}^{n_u}_{C}$ and 
$W\in\mathbb{K}^{n_w}_{C}$. Furthermore, $U$ has a non-empty interior.
\end{assumption}

\begin{definition}
\label{def::RFIT}
The set-valued function $Y: [t_1,t_2] \to \Pi(\mathbb R^{n_x})$ is called a RFIT for \eqref{eq::ODE} on $[t_1,t_2]$, if there exists an integrable feedback control law $\mu: [t_1,t_2] \times \mathbb R^{n_x} \to U$ such that any solution of the controlled system
\begin{align*}
\forall t \in [t_1,t_2], \quad \dot x(t) = f(x(t),w(t)) + G(x(t))\mu(t,x(t))\,,
\end{align*}
with $x(t) \in Y(t)$, satisfies $x(t') \in Y(t')$ for all $t,t' \in [t_1,t_2]$ with $t' \geq t$ and all $w\in \mathbb W$.
\end{definition}

Our focus throughout the paper is on a tube-based robust MPC approach, whereby the following optimization problems are solved in a receding horizon manner:
\begin{equation}
\label{eq::tubeMPC}
\begin{aligned}
\inf_{Y\in\mathcal Y}\ & \int_t^{t+T} \ell(Y(\tau)) \, \mathrm{d}\tau\\
\text{s.t.}\quad & \forall \tau\in[t,t+T], \quad Y(\tau) \subseteq F_x \\
& Y(t) = \{ \hat x_t \} \,,
\end{aligned}
\end{equation}
where $\mathcal Y$ denotes the set of all RFITs for \eqref{eq::ODE} on $[t,t+T]$; $\ell: \Pi(\mathbb R^{n_x}) \to \mathbb R$ is the objective of the MPC controller; the feasibility set $F_x$ is a subset of $\mathbb R^{n_x}$; and $\hat x_t$ is the state measurement at $t$, assumed to be noise free. 

Observe that optimizing over the tube $Y$ in problem~\eqref{eq::tubeMPC} is equivalent to optimizing over a feedback control policy $\mu$, since every $Y$ is generated by at least one $\mu$ according to Definition~\ref{def::RFIT}. This also makes the link with standard MPC formulations, where the optimization is over the (open-loop) control trajectory.

The following analysis aims to develop a tractable computational approach to addressing the tube-based robust MPC problem~\eqref{eq::tubeMPC}. For simplicity, computational delays are not taken into account in this analysis. Specifically, given any feedback control policy $\mu(t,x)$ keeping the response $x$ in an optimal RFIT $Y^*$---e.g., as found from the repeated solution of \eqref{eq::tubeMPC} in a receding horizon manner---we assume that the control $u(t)=\mu(t,\hat x_t)$ is fed back into the system instantaneously.

\section{Characterization of Robust Forward Invariant Tubes}
\label{sec::minmaxDI}

This section presents sufficient conditions for a convex tube to be a RFIT for the nonlinear input-affine control system \eqref{eq::ODE}, under the following generic assumption:

\begin{assumption}
\label{ass::regularityL}
The function $f$ is jointly continuous in $x,w$ and locally Lipschitz-continuous in $x$. Moreover, the function $G$ is continuously differentiable.
\end{assumption}

The derivation builds upon a recent result for computing enclosures of the reachable set of uncertain ODEs \cite{Villanueva2014}. For a given control $u\in\mathbb{U}$ and a given set of initial states $X_1\in\mathbb{K}_C^{n_x}$ at $t_1$, we denote the reachable set of \eqref{eq::ODE} at $t_2>t_1$ as:
\begin{equation*}
X(t_2) := \left\{\xi\in\mathbb{R}^{n_x} \middle|
\begin{aligned}
& \exists x\in \mathbb{L}^{n_x}, \exists w\in\mathbb{W}:\\
& \operatorname{a.e.}\ t\in[t_1,t_2],\\
& \ \ \dot x(t)= g(x(t),u(t),w(t))\\
& x(t_1) \in X_1\,,\ \ x(t_2) = \xi
\end{aligned}
\right\}\, .
\end{equation*}
For notational convenience, we also define the set-valued function $\Gamma_{g}: \mathbb R^{n_u}\times\mathbb R^{n_x}\times\mathbb K^{n_x}_C \to \mathbb K^{n_x}$ associated with the right-hand-side function $g$ in \eqref{eq::ODE} as:
\begin{equation*}
\Gamma_{g}(\nu,c,Z) :=  
\left\{ g(\xi,\nu,\omega) \,\middle|\,
\begin{aligned}
\tr{c}\xi &= V[Z](c) \\
\xi &\in Z \\
\omega &\in W
\end{aligned}
\right\}\, .
\end{equation*}
The following theorem is adapted from~\cite[Theorem~3 \& Remark 2]{Villanueva2014} for the class of controlled dynamic systems of interest.

\begin{theorem}
\label{thm::GDI}
Consider the uncertain dynamic system~\eqref{eq::ODE} with initial condition $x(t_1)\in X_1$, with $X_1\in\mathbb K^{n_x}_{C}$, and a given control $u\in\mathbb{U}$, and let Assumptions~\ref{ass::compact} and~\ref{ass::regularityL} hold. Let $Y: [t_1,t_2]\rightarrow \mathbb K_{C}^{n_x}$ be a set-valued function such that
\begin{enumerate}
\item the function $V[Y(\cdot)](c)$ is, for all $c\in\mathbb R^{n_x}$, Lipschitz-continuous on $[t_1,t_2]$, and 
\item the set-valued function $Y$ satisfies, for all $c\in\mathbb R^{n_x}$, the differential inequality
\begin{align*}
& \operatorname{a.e.}\ t\in[t_1,t_2],\\
& \quad \dot{V}[Y(t)](c) \geq V[\Gamma_{g}(u(t),c,Y(t))](c)\\
& \text{with} \ \ V[Y(t_1)](c) \geq V[X_1](c) \; .
\end{align*}
\end{enumerate}
Then, $Y$ is an enclosure of the reachable tube of \eqref{eq::ODE}, i.e. $Y(t)\supseteq X(t)$ for all $t\in[t_1,t_2]$.
\end{theorem}

The following theorem sets the basis for the tube-based MPC methods that are proposed in the paper. Unlike Theorem~\ref{thm::GDI}, the control policy $u$ is not given, but chosen in the set $\mathbb U$ of admissible controllers in order to reduce the cross-section of the tube, while accounting for every possible realization of the exogenous disturbance $w\in \mathbb W$. These sufficient conditions come in the form of a min-max differential inequality (DI), which describes the convex cross-sections of a RFIT in terms of their support functions. 

\begin{theorem}
\label{thm::minmaxDI}
Consider the uncertain dynamic system~\eqref{eq::ODE}, and let Assumptions~\ref{ass::compact} and~\ref{ass::regularityL} hold. Let $Y: [t_1,t_2]\rightarrow \mathbb K_{C}^{n_x}$ be a set-valued function such that
\begin{enumerate}
\item the function $V[Y(\cdot)](c)$ is, for all $c \in \mathbb R^{n_x}$, Lipschitz-continuous on $[t_1,t_2]$, and
\item the set-valued function $Y$ satisfies, for all $c \in \mathbb R^{n_x}$, the differential inequality
\begin{equation}
\begin{aligned}
& \operatorname{a.e.}\ t\in[t_1,t_2],\\
& \quad \dot{V}[Y(t)](c) \geq \min_{\nu\in U} V[\Gamma_{g}(\nu,c,Y(t))](c) \; . \label{eq::minmaxDI-RHS}
\end{aligned}
\end{equation}
\end{enumerate}
Then, $Y$ is a RFIT for all $t\in[t_1,t_2]$.
\end{theorem} 
\begin{proof}
See Appendix~\ref{app::minmaxDI}.
\end{proof}

The following corollary is a direct side-product of the proof of Theorem~\ref{thm::minmaxDI}.

\begin{corollary}
\label{cor::explicitfeedbackY}
Let the set-valued function $Y:[t_1,t_2]\to \mathbb K_{C}^{n_x}$ satisfy the conditions of Theorem~\ref{thm::minmaxDI}. Under the additional regularity conditions that the set of admissible controls $U$ and the tube cross-sections $Y(t)$, for all $t\in[t_1,t_2]$, are smooth and their boundaries have positive curvature everywhere, an explicit feedback control law keeping the uncertain system trajectories within the RFIT is
\begin{subequations}
\label{eq::explicitfeedbackY}
\begin{align}
\mu(t,\xi) &= \mu^{*}_{t}\left(\mathcal{G}_{Y(t)}(\xi)\right)\,,\\
\text{with}\quad \mu^*_t(c) :&= \argmin_{\nu\in U}\tr{c}G\left(\mathcal{G}^{-1}_{Y(t)}(c)\right)\nu\;,
\end{align}
\end{subequations}
\and the inverse Gauss map $\mathcal{G}^{-1}_{Y(t)}$ of $Y(t)$ is given 
by
\begin{equation*}
\mathcal{G}^{-1}_{Y(t)}(c) = \argmax_{\xi\in Y(t)} \tr{c}\xi \;.
\end{equation*}
\end{corollary}

\begin{remark}
The feedback control law given by Eqs.~\eqref{eq::explicitfeedbackY} is not necessarily unique.
\end{remark}

\begin{remark}
It is clear from Eq.~\eqref{eq::RHS1} in Appendix~\ref{app::minmaxDI}, that the construction of the feedback control law~\eqref{eq::explicitfeedbackY} relies heavily on the assumption of a control-affine structure for $g$ as well as the absence of uncertain inputs $w$ in the matrix-valued function $G$.
\end{remark}

Although heavily inspired by set-theoretic methods for the synthesis of model predictive controllers (see~\cite{Rakovic2009} for an introduction), Theorem~\ref{thm::minmaxDI} also provides a constructive approach for nonlinear feedback control laws by exploiting properties at the boundaries of RFITs. This approach has not been exploited so far in the robust MPC literature. 

Checking the sufficient conditions provided by Theorem~\ref{thm::minmaxDI} for an arbitrary convex set-valued function may prove computationally challenging in general. Nevertheless, the min-max differential inequality can be checked constructively for certain parameterizations of the tube cross-sections, as shown for ellipsoidal tubes next.

\section{Ellipsoidal Robust Forward Invariant Tubes}
\label{sec::E-RFITs}

This section derives computationally tractable conditions for checking whether a particular set-valued function $Y$ is a RFIT for the dynamic system~\eqref{eq::ODE}. The focus is on tubes with ellipsoidal cross-sections, given by
\begin{equation}
\label{eq::elltube}
\begin{aligned}
Y(t) = \mathcal E(q_x(t),Q_x(t))\,,
\end{aligned}
\end{equation}
where $q_x(t)\in\mathbb R^{n_x}$ and $Q_x(t)\in\mathbb S^{n_x}_{+}$ denote the center and shape matrix of the tube, pointwise in time. Moreover, we make the following additional assumptions:

\begin{assumption}
\label{ass::ellsets}
There exist pairs $(q_w,Q_w)\in\mathbb R^{n_w}\times\mathbb S^{n_w}_{+}$ and
$(q_u,Q_u)\in\mathbb R^{n_u}\times\mathbb S^{n_u}_{+}$ such that
$\mathcal E(q_w,Q_w)\supseteq W$ and $\mathcal E(q_u,Q_u) \subseteq U$.
\end{assumption}

\begin{assumption}
\label{ass::regularityC2}
The functions $f$ and $G$ are twice continuously differentiable in all of their arguments. 
\end{assumption}

The following construction of ellipsoidal tubes is based on Theorem~\ref{thm::minmaxDI} and uses the same ideas as the construction of ellipsoidal bounds for uncertain ODEs based on Theorem~\ref{thm::GDI}; see, e.g., \cite{Kurzhanski1993,Houska2012,Villanueva2014}. The control $u$, disturbance $w$ and state $x$ are decomposed into their nominal and perturbed components as
\begin{align*}
u(t) &= q_u + \delta_u(t)\\
w(t) &= q_w + \delta_w(t)\\
x(t) &= q_x(t) + \delta_x(t)\,,\\
\end{align*}
where $q_x$ satisfies the ODE
\begin{align*}
\dot q_x(t) =& f(q_x(t),q_w) + G(q_x(t))\,u_x(t) \,,
\end{align*} 
for a reference control $u_x(t)\in\mathcal{E}(q_{u},Q_{u})$. It follows that the perturbed state component $\delta_x$ satisfies the ODE
\begin{equation}
\label{eq::odedecomposition}
\begin{aligned}
\dot \delta_x(t) = &\phantom{+}A(q_x(t)) \delta_x(t) + B(q_{x}(t)) \delta_w(t)\\
&+ G(q_x(t) + \delta_x(t)) \delta_u(t) \\ 
&+ n(t,\delta_x(t),\delta_w(t),\delta_u(t))\,.
\end{aligned}
\end{equation}
Here, the function $n:\mathbb{R}\times\mathbb R^{n_x}\times \mathbb R^{n_w}\times \mathbb R^{n_u}\to \mathbb R^{n_x}$ is defined in such a way that \eqref{eq::odedecomposition} is equivalent to \eqref{eq::ODE} and 
\begin{align*}
A(q_x(t)) &:= \frac{\partial f}{\partial x} (q_x(t),q_w) + \frac{\partial G}{\partial x}(q_x(t)) u_x(t) \,, \\ 
B(q_x(t)) &:= \frac{\partial f}{\partial w} (q_x(t),q_w) \,.
\end{align*}

A number of remarks are in order. Since the central path $q_{x}$ corresponds to the nominal state, $u_x$ can be understood as the control input that would be applied if no uncertainty were affecting the system. Moreover, the decomposition of the right-hand side per~\eqref{eq::odedecomposition} is valid with any integrable functions $A$ and $B$ of suitable dimensions, as long as $n$ is chosen in an appropriate manner. For instance, if $A$ and $B$ are constructed through a first-order Taylor expansion, $n$ is given by the remainder function per Taylor's theorem.

The present tube construction relies on the existence of an inner approximation of $\mathcal{E}(q_u,Q_u)$ centered at $u_x(t)$, as given by the following lemma.
\begin{lemma}
\label{lem::innercontrol}
For any reference control $u_x(t) \in \mathcal E(q_u,Q_u)$, any function $\gamma: \mathbb R \to (0,1]$, and any matrix-valued function $R_{u}:\mathbb{R}\to\mathbb{S}^{n_{u}}_{+}$ such that $R_u(t) \succeq 0$ and
\begin{equation}
\begin{aligned}
R_{u}(t) &=\hphantom{+} [1-\gamma(t)] Q_{u}  \\ 
&\hphantom{=}+ \left[1-\gamma(t)^{-1}\right] [u_{x}(t)-q_{u}]\tr{[u_{x}(t)-q_{u}]}
\end{aligned}
\end{equation}
we have $\mathcal{E}(u_x(t),R_u(t)) \subseteq \mathcal{E}(q_u,Q_u)$ for all $t \in \mathbb R$.
\end{lemma}

\begin{proof}
See Appendix~\ref{app::ellinnerapprox}.
\end{proof}

We also introduce the following technical assumptions regarding the control constraint set and the nonlinearities in the functions $G$ and $n$.
\begin{assumption}
\label{ass::nonlinearityn}
There exists a nonlinearity bounder $\Omega_{n}:\mathbb R^{n_x}\times\mathbb S^{n_x}_+\to\mathbb S^{n_x}_+$ for the function $n$ such that
\begin{equation*}
n(t,\xi,\omega,\nu) \in \mathcal E(\Omega_{n}(q_x(t),Q_x(t))) \,,
\end{equation*}
for all $t\in[t_1,t_2]$, all $\xi\in\mathcal E(Q_x(t))$, all $\omega\in\mathcal E(Q_w)$, and all $\nu\in\mathcal E(Q_u)$.
\end{assumption}

\begin{assumption}
\label{ass::nonlinearityG}
There exists a nonlinearity bounder $\Omega_{G}:\mathbb R^{n_x}\times \mathbb S^{n_x}_+\times\mathbb S^{n_u}_{+}\times\mathbb R^{n_x\times n_u}\to\mathbb S^{n_x}_+$ such that
\begin{equation*}
\begin{aligned}
&\Omega_{G}(q_x(t),Q_x(t),R_{u}(t),S_{0})\\
&\succeq  Q_x^\frac{1}{2}(t)S_{0}R_u^\frac{1}{2}(t)\tr{G(\xi)} + G(\xi)R_u^\frac{1}{2}(t)\tr{S_0}Q_x^\frac{1}{2}(t) \\
& - Q_x^\frac{1}{2}(t)S_{0}R_u^\frac{1}{2}(t)\tr{G(q_x(t))} - G(q_x(t))R_u^\frac{1}{2}(t)\tr{S_{0}}Q_x^\frac{1}{2}(t) \,,
\end{aligned}
\end{equation*}
for all $t\in[t_1,t_2]$, all $\xi\in\mathcal E(q_x(t),Q_x(t))$ and all $S_{0}\in\mathbb R^{n_x\times n_u}$ with $S_{0}\tr{S_{0}}\preceq I$, where $R_{u}$ is constructed as in Lemma~\ref{lem::innercontrol} such that $\mathcal{E}(u_x(t),R_u(t)) \subseteq \mathcal{E}(q_u,Q_u)$.
\end{assumption}

Sufficient conditions for a tube with ellipsoidal cross-section to be a RFIT for system \eqref{eq::ODE} are stated in the following theorem. For notational convenience, we introduce the matrix-valued function $\Phi_g:\mathbb R^{n_x}\times\mathbb S^{n_x}_+\times \mathbb R^{n_x\times n_u} \times \mathbb{S}^{n_u}_{+}\times \mathbb R_{++}\times\mathbb R_{++} \to \mathbb S^{n_x}_+$ associated with the right-hand-side function $g$ in \eqref{eq::ODE} as:
\begin{equation*}
\begin{alignedat}{1}
&\Phi_{g}(q_x(t),Q_x(t),S_0,R_{u}(t),\lambda_0,\kappa_0) := A(q_x(t))Q_x(t) \\
& \quad + Q_x(t)\tr{A(q_x(t))} + Q_x^\frac{1}{2}(t)S_0 R_u^\frac{1}{2}(t)\tr{G(q_x(t))}\\
& \quad  + G(q_x(t))R_u^\frac{1}{2}(t)\tr{S_0}Q_x^\frac{1}{2}(t) + \left(\frac{1}{\lambda_0} 
+ \frac{1}{\kappa_0} \right) Q_x(t) \\
& \quad  + \Omega_{G}(q_x(t),Q_x(t),R_u(t),S_0) \\
& \quad  + \lambda_0 B(q_x(t))Q_w\tr{B(q_x(t))} + \kappa_0 \Omega_{n}(q_x(t),Q_x(t)) \, .
\end{alignedat}
\end{equation*} 
\begin{theorem}
\label{thm::elltube}
Consider the uncertain dynamic system~\eqref{eq::ODE}, and let Assumptions~\ref{ass::ellsets}-\ref{ass::nonlinearityG} hold for a given reference control $u_x\in\mathbb U$ and let $R_u$ be constructed as in Lemma~\ref{lem::innercontrol}. If the functions $Q_x:[t_1,t_2]\to\mathbb S^{n_x}_{+}$ and $q_x:[t_1,t_2]\to\mathbb R^{n_x}$ satisfy 
\begin{align}
\dot q_x(t) =\ & f(q_x(t),q_w) + G(q_x(t))u_x(t) \label{eq::elltube_qx}\\
\dot Q_x(t) \succeq\ & \Phi_g(q_x(t),Q_x(t),S(t),R_{u}(t),\lambda(t),\kappa(t))\,,\label{eq::elltube_Qx}
\end{align}
for some functions $\lambda,\kappa:[t_1,t_2]\to\mathbb R_{++}$ and $S:[t_1,t_2]\to\mathbb R^{n_x\times n_u}$ with $S(t)\tr{S(t)} \preceq I$, then $Y(t):=\mathcal E(q_x(t),Q_x(t))$ describes a RFIT for \eqref{eq::ODE} on $[t_1,t_2]$.
\end{theorem}
\begin{proof}
See Appendix~\ref{app::elltube}. 
\end{proof}

The following corollary is an immediate consequence of the proofs of Theorem~\ref{thm::minmaxDI} (Step S1) and Theorem~\ref{thm::elltube}. 
\begin{corollary}
\label{cor::explicitfeedbackellipsoid}
Let the set-valued function $Y:[t_1,t_2]\to\mathbb{K}^{n_x}_{C}$ with 
$Y(t):=\mathcal{E}(q_{x}(t),Q_{x}(t))$ satisfy the conditions of 
Theorem~\ref{thm::elltube}. Then, an explicit feedback law associated 
with this RFIT is given by
\begin{equation}
\label{eq::explicitfeedbackellipsoid} 
\mu(t,\xi) = 
\begin{cases}
\mu^{*}_{t}(\mathcal{G}_{Y(t)}(\xi)) & \text{if }\; \xi \in \bd Y(t) \\
u_x(t) &\text{otherwise}
\end{cases}\;,
\end{equation}
where $\mathcal G_{Y(t)}$ and $\mathcal G^{-1}_{Y(t)}$ denote the Gauss map of $\mathcal E(q_x(t),Q_x(t))$ and its inverse respectively, i.e.
\begin{align*}
\mathcal{G}_{Y(t)}(\xi) &= \frac{Q^{\dagger}_x(t)(\xi - q_x(t))}
{\left\Vert Q^{\dagger}_x(t)(\xi - q_x(t))\right\Vert_{2}}\;, \\
\mathcal{G}^{-1}_{Y(t)}(c) &= q_x(t)+ \frac{Q_x(t)c}{\sqrt{\tr{c}Q_{x}(t)c}}\;,
\end{align*}
and $\mu_t^{*}$ is given by
\begin{equation*}
\mu^{*}_{t}(c) = u_x(t) - \frac{R_{u}(t)\tr{G\left(\mathcal{G}^{-1}_{Y(t)}(c)\right)}c}{\left\| R_u^{\frac{1}{2}}(t)\tr{G\left(\mathcal{G}^{-1}_{Y(t)}(c)\right)}c\ \right\|_2}\,.
\end{equation*}
\end{corollary}
\begin{remark}
Another feedback law can be obtained by extending the the domain of the Gauss map of an ellipsoid from $\bd \mathcal{E}(q,Q)$ to $\mathcal{E}(q,Q)\setminus\{q\}$, and replacing the condition $\xi\in\bd Y(t)$ with $\xi\neq q_{x}$ in the feedback law~\eqref{eq::explicitfeedbackellipsoid}.
\end{remark}
Depending on the problem at hand, the required nonlinear bounders in Assumptions~\ref{ass::nonlinearityn} and~\ref{ass::nonlinearityG} may be constructed either symbolically, as proposed in~\cite{Houska2012}, or numerically, e.g. using tools from interval analysis~\cite{Villanueva2014}. A difficulty with the latter approach, however, is that operations performed using usual interval arithmetic are Lipschitz continuous, yet typically nonsmooth. This would impair the use of gradient-based methods for solving the optimization problems. Instead of applying interval analysis directly, Lemma~\ref{lem::frobeniusbound} in Appendix~\ref{app::frobeniusbound} presents a way of constructing a smooth nonlinearity bounder for any twice continuously-differentiable function.

\section{Robust Tube-Based MPC Based on Min-Max Differential Inequalities}
\label{sec::minmaxMPC}

This section discusses how the developments in Sect.~\ref{sec::minmaxDI} and Sect.~\ref{sec::E-RFITs} can be used in the context of robust MPC. Using Theorem~\ref{thm::minmaxDI}, any solution to the following optimization problem turns out to also be a feasible solution to the tube-based MPC problem~\eqref{eq::tubeMPC}, in the case of RFITs with convex cross-sections:
\begin{equation}
\label{eq::tubeMPC_DI}
\begin{alignedat}{3}
&\inf_{Y}\ && \int_t^{t+T} \ell(Y(\tau)){\rm d}\tau \\ 
&\text{s.t.}\quad && \operatorname{a.e. } \tau\in[t,t+T],\ \forall c\in \mathbb R^{n_x},\\
& && \quad \dot{V}[Y(\tau)](c) \geq \min_{\nu\in U}V[\Gamma_{g}(\nu,c,Y(\tau))](c)\\
& &&\forall \tau\in[t,t+T], \quad Y(\tau) \subseteq F_{x}\\
& &&Y(t) = \{ \hat x_t \} \,.
\end{alignedat}
\end{equation}
Notice that~\eqref{eq::tubeMPC_DI} is not a standard optimal control problem, as it embeds semi-infinite differential inequality constraints. However, discretizing this problem leads to a band-structured optimization problem whose complexity scales linearly with respect to the length of the time horizon.

With the results from Theorem~\ref{thm::elltube}, Problem~\eqref{eq::tubeMPC_DI} can be further specialized to the case of tubes with ellipsoidal cross-sections as:
\begin{equation}
\label{eq::tubeMPC_ERFIT}
\begin{alignedat}{3}
&\!\!\!\!\!\!\!\!\inf_{\substack{Q_x,R_u,S,\\q_x,u_x,\gamma,\\\lambda,\kappa}}\ && 
\int_t^{t+T} \ell(\mathcal E(q_x(\tau),Q_x(\tau))){\rm d}\tau \\ 
&\text{s.t.}\quad && \operatorname{a.e.} \tau\in[t,t+T],\\
& && \quad 
\begin{aligned}
\dot q_x(\tau) &= f(q_x(\tau),q_w) + G(q_x(\tau))u_x(\tau),\\
\dot Q_x(\tau) &= \Phi_g(q_x(\tau),Q_x(\tau),S(\tau),R_{u}(\tau),\lambda(\tau),\kappa(\tau)),\\
R_{u}(\tau) &= [1-\gamma(\tau)] Q_{u} \\ & \ + [1-\gamma(\tau)^{-1}] [u_{x}(\tau)-q_{u}]\tr{[u_{x}(\tau)-q_{u}]},\\
\end{aligned}  \\
& &&  q_x(t) = \hat x_t, \quad Q_x(t) = 0, \\
& && \forall \tau\in[t,t+T], \\
& && \quad Q_x(\tau)\succeq 0,\;  R_u(\tau) \succeq 0,\; S(\tau)\tr{S(\tau)}\preceq I, \\
& && \quad  \kappa(\tau)>0, \; \lambda(\tau)>0 ,\; 0<\gamma(\tau)<1, \\
& && \quad \mathcal E(q_x(\tau),Q_x(\tau)) \subseteq F_{x}, \;
u_x(\tau)\in\mathcal{E}(q_u,Q_u)\;. 
\end{alignedat}
\end{equation} 

Observe that~\eqref{eq::tubeMPC_ERFIT} now yields a standard optimal control problem with linear matrix inequality (LMI) constraints. A solution to this problem provides a RFIT in the form $Y(\tau)=\mathcal E(q_x(\tau),Q_x(\tau))$, from which an explicit feedback control law can be derived by applying Corollary~\ref{cor::explicitfeedbackellipsoid}.

A practical implementation of this tube-based MPC scheme calls for the specification of the performance criterion $\ell$ and the feasibility set $F_x$. In the case of tracking control, we may use the so-called generalized rotational inertia of the set $Y(t)$ with respect to a given reference $x_{\rm ref}$~\cite{Houska2015}, defined by:
\begin{equation}
\label{eq::gen-inertia}
\ell(Y(t)) := \frac{\int_{Y(t)} \tr{(x-x_{\rm ref})}D(x-x_{\rm ref}){\rm d}x}{\int_{Y(t)} 1{\rm d}x}\;,
\end{equation}
where $D\in\mathbb S^{n_x}_{++}$ is any weighting matrix. In the ellipsoidal case, $Y(t):=\mathcal{E}(q_{x}(t),Q_x(t))$, we have~\cite[Appendix~C]{Villanueva2016}
\begin{align*}
\ell(\mathcal E(q_x(t),Q_x(t))) =\ & \tr{(q_x(t) - x_{\rm ref})}D(q_x(t) - x_{\rm ref})\\
& + \frac{\operatorname{Tr}(DQ_x(t))}{n_x+2}\,.
\end{align*}
Regarding the feasible set, we may consider linear state constraints of the form
\begin{equation*}
F_x := \left\{ x\in\mathbb R^{n_x} \;\middle|\; \tr{h_i}x\leq \eta_i,\ i=1,
\ldots,n_h \right\}\,,
\end{equation*}
with $h_i\in\mathbb R^{n_x}$ and $\eta_i\in\mathbb R$. In the ellipsoidal case, the feasibility constraint $\mathcal E(q_x(\tau),Q_x(\tau)) \subseteq F_{x}$ can be rewritten as~\cite{Kurzhanski1997}:
\begin{equation*}
\forall \tau\in[t,t+T], \quad \tr{h_i}q_x(\tau) + \sqrt{\tr{h_i}Q_x(\tau) h_i} \leq \eta_i \,.
\end{equation*}
One of the main issues in robust MPC is ensuring recursive feasibility, namely the ability to find, for every possible initial state, a feasible state at every time along the closed-loop trajectory. This requirement can be addressed by adding the following constraint to the optimization problem~\eqref{eq::tubeMPC_ERFIT}:
\begin{equation}
\label{eq::terminalfeasibility}
Y(t+T)  \subseteq Y_{\rm ref}\,,
\end{equation}
where $Y_{\rm ref}\subseteq F_x$ is a robust forward invariant set, i.e. a time-invariant RFIT. If $Y_{\rm ref}$ satisfies Definition~\ref{def::RFIT} on any time interval, then the sets $\{\mu(t+T,x(t+T))|x(t+T)\in Y(t+T) \}\in U$ will remain non-empty by construction, and the MPC procedure discussed previously is indeed recursively feasible. 
This recursive feasibility condition is satisfied, if
\begin{equation*}
\Phi_g(x_{\rm ref},Q_{\rm ref},S_{\rm ref},Q_u,\lambda_{\rm ref},\kappa_{\rm ref}) \preceq 0
\end{equation*}
for some scalar $\lambda_{\rm ref},\kappa_{\rm ref}\in\mathbb R_{++}$ and some matrix $S_{\rm ref}\in\mathbb R^{n_x\times n_u}$ with $S_{\rm ref}\tr{S_{\rm ref}}\preceq I$. For instance, one such matrix $Q_{\rm ref}$ can be found by solving the following optimization problem:
\begin{equation}
\label{eq::ellipsoidalbarFx}
\begin{aligned}
\inf_{\substack{Q_{\rm ref},\lambda_{\rm ref},\\\kappa_{\rm ref},S_{\rm ref}}}\ & \operatorname{Tr}\left(Q_{\rm ref}\right) \\ 
\text{s.t.}\quad & \Phi_{g}(x_{\rm ref},Q_{\rm ref},S_{\rm ref},Q_u,\lambda_{\rm ref},\kappa_{\rm ref}) \preceq 0 \\
& Q_{\rm ref}\in\mathbb S_+^{n_x}, \; \lambda_{\rm ref},\;\kappa_{\rm ref}>0, \; S_{\rm ref}\tr{S_{\rm ref}} \preceq I\,.
\end{aligned}
\end{equation}
The following section presents an application of the ellipsoidal approach of tube-based MPC on a numerical case-study.

\section{Numerical Case Study}
\label{sec::numerics}

We consider a spring-mass-damper system \cite{Rubagotti2009} given by
\begin{equation*}
\label{eq::smd}
\underbrace{\begin{pmatrix}
\dot{x}_1(t) \\
\dot{x}_2(t) 
\end{pmatrix}}_{\displaystyle\dot{x}(t)} =
\underbrace{\begin{pmatrix}
x_2(t) + w_1(t) \\
-\frac{k(x)x_1(t)}{M} - \frac{h_dx_2(t)}{M} + \frac{w_2(t)}{M} 
\end{pmatrix}}_{\displaystyle f(x(t),w(t))}
+ 
\underbrace{\begin{pmatrix}
0 \\
\frac{1}{M} 
\end{pmatrix}}_{\displaystyle G(x(t))}u(t) \; ,
\end{equation*}
where $x_1$ and $x_2$ denote the displacement of the cart with respect to the equilibrium position $\rm [m]$ and its velocity $\rm [m/s]$, respectively; $M$ is the mass of the cart; $k(x):=k_0\exp{(-x_1)}$, the stiffness of the spring; and $h_d$, the damping factor. The values of the parameters are $M=1~{\rm kg}$, $k_0=0.33~{\rm N/m}$ and $h_d=1.1~{\rm Ns/m}$. 

Bounds for the disturbance and the control sets are given by the ellipsoids $\mathcal{E}(Q_w)\in\mathbb{K}^{2}_{C}$ and $\mathcal{E}(Q_u)\in\mathbb{K}_{C}$, with $Q_w = \operatorname{diag}(10^{-2}~{\rm m^2/s^2},\; 0.25~{\rm N^2})$ and $Q_u = 36\ {\rm N^2}$. The length of the prediction horizon is set to $T = 10\ {\rm s}$, and the initial state of the system is $x_{\rm start} = \tr{\left( 0.7~{\rm m}, 0.7~{\rm m/s} \right)}$. 

\begin{figure}[h!]
\centering
\includegraphics[width = 0.9\columnwidth]{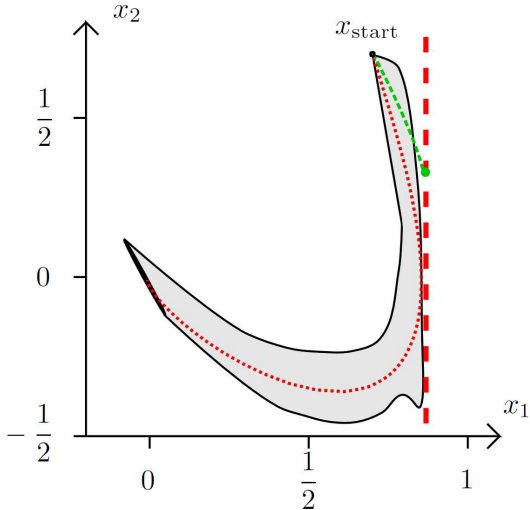}
\caption{\label{fig::rmpc5}Comparison of the robust (ellipsoidal) tube-based controller with a certainty-equivalent model predictive controller. The plot shows the optimal ellipsoidal RFIT (grey area) for $\hat{x}_{t} = x_{\rm start}$ as well as a nominal trajectory (red line) and a disturbed closed-loop trajectory (green line) for the certainty-equivalent controller.}
\end{figure}

The optimization problem in the tube-based MPC controller is based on \eqref{eq::tubeMPC_ERFIT} and involves minimizing the functional
\begin{equation*}
\int_0^T \left(\left\| q_x(t) \right\|_2^2 + \frac{1}{4} \, \mathrm{Tr}\left( Q_x(t) \right) + u_x(t)^2 \,\right) \mathrm{d}t \; .
\end{equation*}
This cost corresponds to the generalized rotational inertia, except for the term $u_x(t)^2$ which can be interpreted as a control regularization. Moreover, a state constraint is enforced, so that $\mathcal E(q_x(\tau),Q_x(\tau)) \subseteq F_{x} := \{ x \mid x_1 \leq 0.85 \}$.

Problem~\eqref{eq::tubeMPC_ERFIT} is solved numerically using the optimal control sofware ACADO~\cite{Houska2011},\footnote{Since ACADO Toolkit does not support LMI constraints, our implementation substitutes the LMI constraints in \eqref{eq::tubeMPC_ERFIT} with equivalent standard (nonlinear) state constraints using Schur complement techniques~\cite{Boyd1994}.} using a piecewise constant control discretization on $40$ equidistant intervals. All the nonlinearity bounders are constructed using the technique in Appendix~\ref{app::frobeniusbound}.

Fig.~\ref{fig::rmpc5} compares the optimal ellipsoidal RFIT (grey area) with closed-loop trajectories for the nominal system (red line) and a system subject to a random disturbance taking values in $\mathcal{E}(Q_w)$ (green line) for a certainty-equivalent MPC controller. The latter minimizes the tracking objective 
\begin{equation*}
\int_{0}^{T}\left(\Vert x(t) \Vert^{2}_{2} + \Vert u(t)\Vert^{2}_{2} \right) \mathrm{d}t\;, 
\end{equation*}
and is implemented in ACADO using the same parameter values and state constraint as the robust tube-based MPC controller above. Notice that in the case where no uncertainty is present, the certainty-equivalent MPC controller performs as expected---although it touches the state constraint, it is able to steer the state to a neighbourhood of the origin without violating it. In contrast, when the system is subject to disturbances this controller fails in about $50\%$ of  the uncertainty scenarios, after causing a constraint violation.

The results of the tube-based MPC controller are shown in Fig.~\ref{fig::rmpc2}. As expected, the controller steers the nominal state (center of the RFIT) close to the origin at $t=10$. In order to prevent violation of the path constraint against all the possible uncertainty scenarios, the controller rotates the ellipsoidal cross-sections of the RFIT quite drastically initially. Moreover, solving a conservative approximation of the min-max differential inequality appears to have a small adverse effect on the controller's performance in this simple case study.

\begin{figure*}
\centering
\begin{minipage}[b]{.49\textwidth}
\includegraphics[width = 0.9\textwidth]{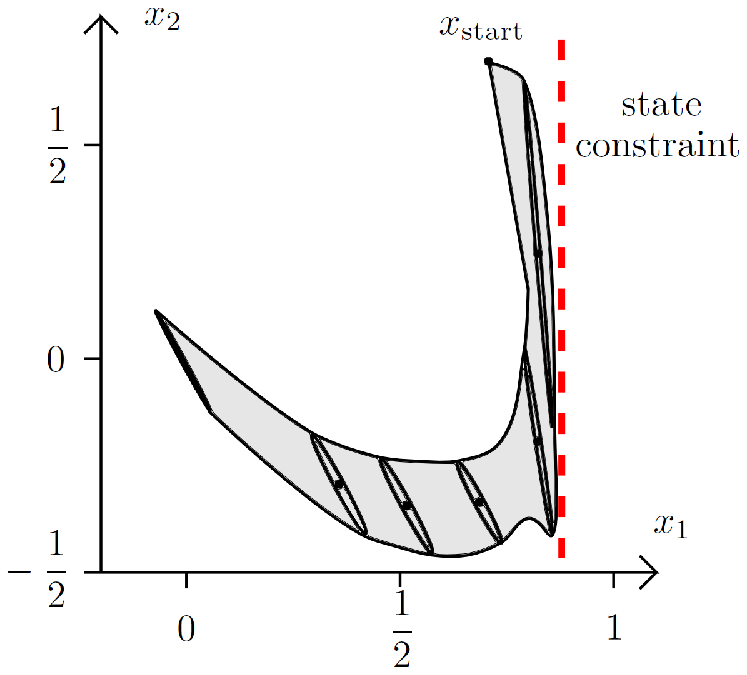}
\end{minipage}%
\begin{minipage}[b]{.49\textwidth}
\includegraphics[width = 0.9\textwidth]{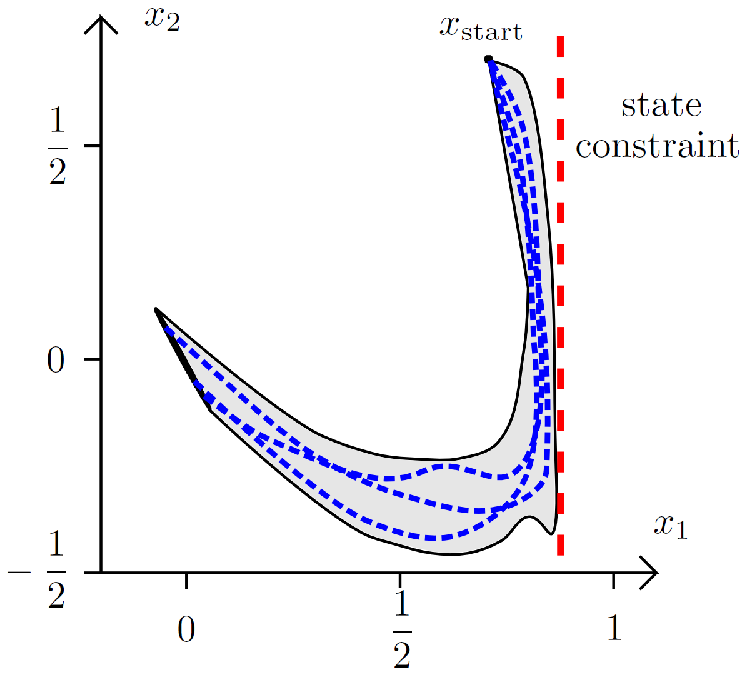}
\end{minipage}
\caption{\label{fig::rmpc2}The optimal ellipsoidal RFIT for $\hat x_t = x_\mathrm{start}$ (grey area). The red line shows the state constraint $\mathbb F_x = \{ x \mid x_1 \leq 0.85 \}$. {\em Left:} Selected ellipsoidal cross-sections for $t \in \{ 1/4, 3/4, 5/4, 7/4, 9/4, 10 \}$. {\em Right:} Selected trajectories for three uncertainty realizations in dotted lines.}
\end{figure*}

\section{Conclusions}
\label{sec::conclusion}

A novel approach to tube-based robust MPC has been proposed for control-affine nonlinear systems, which relies on a min-max differential inequality formulation in order to provide sufficient conditions for a time-varying convex set-valued function to be a RFIT. Unlike other robust MPC approaches, the procedure based on this differential inequality does not call for any particular parameterization of the feedback control law, while benefiting from having linear complexity with respect to the time horizon. Another benefit of the proposed approach is that a semi-explicit representation of a feedback control law may be obtained as a side-product of the RFIT propagation under mild conditions, namely when the RFIT cross-sections and the control sets are smooth with positive curvature. This property has been exploited to devise a practical implementation involving tubes with ellipsoidal cross-sections. This ellipsoidal tube-based MPC approach was tested for a spring-mass-damper system. In contrast to the certainty-equivalent model predictive controller it guarantees feasibility for all uncertainty scenarios.

\small
\begin{ack}                               
This paper is based upon work supported by the Engineering and Physical Sciences Research Council (EPSRC) under Grant EP/J006572/1. Financial support from Marie Curie Career Integration Grant PCIG09-GA-2011-293953 and from the Centre of Process Systems Engineering (CPSE) of Imperial College is gratefully acknowledged. M.E.V. thanks CONACYT for doctoral scholarship. This research was supported by the EU via ERC-HIGHWIND (259 166), FP7-ITN-TEMPO (607 957), and H2020-ITN-AWESCO (642 682). Rien Quirynen holds a research fellowship by the FWO. Support by Freiburg University in form of a guest professorship of the last author at the Freiburg Institute for Advanced Studies (FRIAS) in 2014 is gratefully acknowledged.
\end{ack}
%


\bibliographystyle{plain}        
\bibliography{minmaxDI_ArXiv}

\appendix
\section{Proof of Theorem~\ref{thm::minmaxDI}}
\label{app::minmaxDI}

The proof of Theorem~\ref{thm::minmaxDI} proceeds in two steps. In the first step (S1), we establish the results under the following auxiliary assumptions:
\begin{enumerate}
\renewcommand{\theenumi}{A\arabic{enumi}}
\item \label{ass::Uextra}The set of admissible controls $U$ is smooth with positive curvature;
\item \label{ass::Yextra}The pointwise-in-time cross-sections $Y(t)$ of the tube $Y$ are smooth with positive curvatures at each $t\in[t_1,t_2]$.
\end{enumerate}
In the second step (S2), we argue that the result still holds by removing these extra assumptions. 

\paragraph*{S1}
We start by noting that since the inequality \eqref{eq::minmaxDI-RHS} is invariant under scaling of the directions $c\in\mathbb{R}^{n_x}$, it is sufficient to consider those directions $c$ with $\tr{c}c = 1$, namely $c\in\mathcal{S}^{n_x-1}$.

It follows from Assumption~\ref{ass::Yextra} that the Gauss map $\mathcal{G}_{Y(t)}:\bd Y(t)\to \mathcal{S}^{n_x-1}$ is a diffeomorphism~\cite{Sacksteder1960}. For all $c\in\mathcal{S}^{n_x-1}$ and all $t\in[t_1,t_2]$, the inverse Gauss map values $\mathcal{G}_{Y(t)}^{-1}(c)$ correspond to the elements of the singletons 
\begin{align*}
\Psi_t(c)\ :=&\ \left\{ \xi \in \mathbb R^{n_x} \middle|
\begin{aligned}
\tr{c}\xi &= V[Y(t)](c) \\
 \xi &\in Y(t) 
\end{aligned} \right\} \\
=&\ \argmax_{\xi\in Y(t)} \tr{c}\xi\,.
\end{align*}
In particular, we have
\begin{equation}
\label{eq::RHS1}
\begin{aligned}
\min_{\nu\in U} & \; V[\Gamma_{g}(\nu,c,Y(t))](c)\\ 
& = \max_{\substack{\xi\in\Psi_t(c),\\\omega\in W}} \tr{c}f(\xi,\omega) + \min_{\nu\in U} \max_{\xi\in\Psi_t(c)} \tr{c}G(\xi)\nu\\ 
& =  \max_{\omega\in W} \tr{c}f(\mathcal{G}_{Y(t)}^{-1}(c),\omega) + \min_{\nu\in U} \tr{c}G(\mathcal{G}_{Y(t)}^{-1}(c))\nu\,.
\end{aligned}
\end{equation}
Moreover, by Lipschitz continuity of $V[Y(\cdot)](c)$ on $[t_1,t_2]$, the functions $\mathcal{G}_{Y(\cdot)}^{-1}(c):[t_1,t_2]\to\bd Y(t)$ are continuous for each $c\in\mathcal{S}^{n_x-1}$.

Next, we focus on the minimization subproblem in the right-hand side of \eqref{eq::RHS1}. By continuity of $\mathcal{G}_{Y(t)}^{-1}$ and $G$ (Assumption~\ref{ass::regularityL}) and by compactness of $U$ (Assumption~\ref{ass::compact}), the sets $\argmin_{\nu\in U} \tr{c}G(\mathcal{G}_{Y(t)}^{-1}(c))\nu$ are singletons for all $c\in\mathcal{S}^{n_x-1}$ and all $t\in[t_1,t_2]$, and we can define the function $\mu^*_t$ as
\begin{equation*}
\mu^*_t(c) := \argmin_{\nu\in U} \tr{c}G(\mathcal{G}_{Y(t)}^{-1}(c))\nu \,.
\end{equation*}
Since $\mu^*_t(c)$ is always attained at the boundary of $U$, it follows by Assumption~\ref{ass::Uextra}, by continuous differentiability of $\mathcal{G}_{Y(t)}^{-1}$ and $G$ (Assumption~\ref{ass::regularityL}), and from sensitivity theory~\cite{Fiacco1983} that $\mu^*_t(\cdot)$ is continuously differentiable on $\mathcal{S}^{n_x-1}$, for each $t\in[t_1,t_2]$. Moreover, the function $\mu^*_{(\cdot)}(c):[t_1,t_2]\to U$ is continuous for each $c\in\mathcal{S}^{n_x-1}$.

The result follows from the application of Theorem~3.1 in \cite{Villanueva2014} (see also Theorem~\ref{thm::GDI} herein) to the auxiliary ODE
\begin{equation*}
\dot x(t) = f(x(t),w(t)) + G(x(t))\mu(t,x(t)) \;,
\end{equation*}
with $\mu(t,\xi) := \mu^*_t(\mathcal{G}_{Y(t)}(\xi))$. In particular, $\mu$ provides a feedback control law for the RFIT $Y$ under the auxiliary Assumptions~\ref{ass::Uextra} and~\ref{ass::Yextra}.

\paragraph*{S2}
In the case that certain tube cross-sections $Y(t)$ or the control constraint set $U$ fail to be smooth with positive curvature on $[t_1,t_2]$, we can---due to Assumption~\ref{ass::compact}---always construct a family of set-valued functions $Y_{\epsilon}:[t_1,t_2]\to \mathbb K^{n_x}_{C}$ as well as a family of compact sets $U_{\epsilon} \subseteq U$ with smooth boundary and positive curvature such that the following statements hold for all $\epsilon >0$:
\begin{enumerate}
\item $Y_{\epsilon}(t) \supseteq Y(t)$ for all $t \in [t_1,t_2]$,\quad $U_\epsilon \subseteq U$.
\item There exists a continuous function $\alpha: \mathbb R_+ \to \mathbb R_+$ with $\alpha(0)=0$ such that
\begin{equation*}
\begin{aligned}
&d_\mathrm{H}(Y_\epsilon(t),Y(t)) \leq \alpha(\epsilon), \quad d_\mathrm{H}(U_\epsilon,U) \leq \alpha(\epsilon), \\
&\text{and}\quad \dot{V}[Y_\epsilon(t)](c) \geq \dot{V}[Y(t)](c) + L \alpha(\epsilon)
\end{aligned}
\end{equation*}
for all $t \in [t_1,t_2]$ with $L := \frac{1}{2(t_2-t_1)}$. In particular, $L$ can be made arbitrarily large by choosing $t_2-t_1$ sufficiently small.
\end{enumerate}
The existence of such outer approximations has been proven in \cite[see Lemma~1 and Lemma~2 in the appendix]{Villanueva2014}. This way, the result follows from the application of the procedure in S1 above and taking the limit as $\epsilon\to 0$ by invoking a continuity argument. In detail, we can always choose $t_2>t_1$ so that
\begin{align}
\dot{V}[Y_\epsilon(t)](c) &\geq \dot{V}[Y(t)](c) + L \alpha(\epsilon) \notag \\[0.1cm]
&\geq \min_{\nu\in U} V[\Gamma_{g}(\nu,c,Y(t))](c) + L \alpha(\epsilon) \notag \\[0.1cm]
&\geq \min_{\nu\in U_\epsilon} V[\Gamma_{g}(\nu,c,Y_\epsilon(t))](c)
\end{align}
for a sufficiently small $\epsilon > 0$. Thus, it follows from step S1 that $Y_\epsilon$ is a RFIT for all sufficiently small $\epsilon > 0$. The final technical difficulty involves analyzing the limit behavior of the sequence 
\begin{equation*}
\begin{aligned}
\mu_{\epsilon}(t,\xi) &:= \mu^*_{t,\epsilon}(\mathcal{G}_{Y_{\epsilon}(t)}(\xi))\\[0.1cm]
\text{with}\;\mu^*_{t,\epsilon}(c) &:= \argmin_{\nu\in U_\epsilon} \tr{c}G(\mathcal{G}_{Y_\epsilon(t)}^{-1}(c))\nu \; ,
\end{aligned}
\end{equation*}
which may fail to converge  as $\epsilon \to 0$. Since $\mu_{\epsilon}(t,\xi)$ takes values in $U_{\epsilon}$ and the sets $U_{\epsilon}$ converge in the Hausdorff sense to a compact set $U$, the sequence $\mu_{\epsilon}(t,\xi)$ is bounded uniformly with respect to $\epsilon > 0$. Consequently, we can use the Bolzano-Weierstrass theorem to establish the existence of a sequence $\epsilon_1,\epsilon_2,... \in \mathbb R_{+}$ with $\lim_{i \to \infty} \epsilon_i \to 0$ such that the limit
\begin{equation*}
\mu(t,\xi) = \lim_{i \to \infty} \; \mu_{\epsilon_i}(t,\xi)
\end{equation*}
exists. By construction, $\mu(t,\xi)$ is a control law that generates the limit tube $Y$, therefore $Y$ is a RFIT.\qed

\section{Proof of Theorem~\ref{thm::elltube}}
\label{app::elltube} 

In analogy to the proof of Theorem~\ref{thm::minmaxDI}, the following proof proceeds in two steps. In the first step (S1), we establish the results under the following auxiliary assumption:
\begin{enumerate}
\setcounter{enumi}{2}
\renewcommand{\theenumi}{A\arabic{enumi}}
\item \label{ass::Qextra}The shape matrices $Q_u$ and $Q_x(t)$, $t_1\leq t\leq t_2$, are positive definite.
\end{enumerate}
In the second step (S2), we argue that the result still holds by removing this extra assumption.

\paragraph*{S1} The idea in this part of the proof is to show that the conditions~\eqref{eq::elltube_qx}--\eqref{eq::elltube_Qx} imply the min-max differential inequality \eqref{eq::minmaxDI-RHS} for $Y(t):=\mathcal E(q_x(t),Q_x(t))$. 
Using the state decomposition into nominal part~\eqref{eq::elltube_qx} and perturbed part~\eqref{eq::odedecomposition} as well as Assumption~\ref{ass::nonlinearityn}, we want to show that
\begin{equation}
\label{eq::cond1}
\begin{aligned}
 \dot{V}[\mathcal E(Q_x&(t))](c) \\
 & \geq \min_{\nu\in \mathcal{E}(R_u(t))} V[ \Gamma_{g_{\delta_{x}}}\left(  \nu,c, \mathcal{E}(Q_x(t)) \right) ](c) 
\end{aligned}
\end{equation} 
for a.e. $t\in [t_1,t_2]$ and all $c \in \mathbb R^{n_x}$ such that $\tr{c}c=1$. 
Here, the set-valued function $\Gamma_{g_{\delta_{x}}}$ is given by
\begin{align*}
&\Gamma_{g_{\delta_{x}}}(\nu,c,\mathcal{E}(Q_x(t)))\\
 &:= \left\{
 \begin{aligned} 
 &\phantom{+}A(q_x(t))\xi\\
 &+B(q_x(t))\omega_1\\
 &+G(q_x(t)+\xi)\nu\\
 &+\omega_2 
 \end{aligned} \middle| 
 \begin{aligned}
 \tr{c}\xi &= V[\mathcal{E}(Q_x(t))](c) \\
 \xi &\in \mathcal{E}(Q_x(t)) \\
 \omega_1 &\in \mathcal{E}(Q_w)\\
 \omega_2 &\in \mathcal{E}(\Omega_{n}(q_x(t),Q_x(t)))
 \end{aligned} \right\}\,. 
\end{align*} 
Moreover, the controls are optimized over $\mathcal{E}(R_u(t))$ in \eqref{eq::cond1}, since the central path $q_x(t)$ in \eqref{eq::elltube_qx} is evaluated along $u_x(t)\in\mathcal{E}(q_u,Q_u)$, instead of the center $q_u$ of $\mathcal{E}(q_{u},Q_{u})$. Recall that the construction of such an inner ellipsoid $\mathcal{E}(u_{x}(t),R_u(t))\subseteq\mathcal{E}(q_u,Q_u)$ is given by Lemma~\ref{lem::innercontrol}.

Next, consider a family of ellipsoids parameterized by the matrix valued function $Q_x:[t_1,t_2]\to\mathbb{S}_{++}^{n_x}$. 
For each $t\in[t_1,t_2]$ the Gauss map $\mathcal{G}_{\mathcal{E}_{x}(t)}:\bd\mathcal E(Q_x(t))\to \mathcal S^{n_x-1}$ 
is a diffeomorphism under Assumption~\ref{ass::Qextra}, given by:
\begin{align*}
\mathcal{G}_{\mathcal{E}_{x}(t)}(\xi) & := \frac{Q_x^{-1}(t)\xi}{\left\|Q_x^{-1}(t)\xi\right\|_2}\,, & \mathcal{G}_{\mathcal{E}_{x}(t)}^{-1}(c) & 
= \frac{Q_x(t)c}{\sqrt{\tr{c} Q_x(t) c}}\,.
\end{align*}
In particular, the right-hand side of Condition~\eqref{eq::cond1} is given by
\begin{equation*}
\begin{aligned}
&\min_{\nu\in \mathcal{E}(R_u(t))} V[ \Gamma_{g_{\delta_{x}}}\left(  \nu,c, \mathcal{E}(Q_x(t)) \right) ](c) \\
& \ = \tr{c}A(q_x(t)) \mathcal{G}_{\mathcal{E}_{x}(t)}^{-1}(c) + \min_{\nu\in \mathcal E(R_u(t))}\tr{c}G(q_x(t)+\mathcal{G}_{\mathcal{E}_{x}(t)}^{-1}(c))\nu \\
& \ \ \ + \max_{\omega_1 \in \mathcal E(Q_w)} \tr{c} B(q_x(t)) \omega_1 + \max_{\omega_2 \in \mathcal E(\Omega_{n}(q_x(t),Q_x(t)))} \tr{c} \omega_2\,.
\end{aligned}
\end{equation*}
Since, for any matrices $D\in\mathbb R^{n_x\times n_\zeta}$ and $Q_{\zeta}\in\mathbb S_{+}^{n_\zeta}$,
\begin{equation*}
\max_{\zeta}/\min_{\zeta}\left\{ \tr{c}D\zeta \middle| \zeta\in\mathcal{E}(Q_{\zeta}) \right\} =\pm \sqrt{\tr{c}DQ_{\zeta}\tr{D}c}\;,
\end{equation*}
we obtain
\begin{equation*}
\begin{aligned}
&\min_{\nu\in \mathcal{E}(R_u(t))} V[ \Gamma_{g_{\delta_{x}}}\left(  \nu,c, \mathcal{E}(Q_x(t)) \right) ](c) \\
& \ = \tr{c}A(q_x(t)) \mathcal{G}_{\mathcal{E}_{x}(t)}^{-1}(c)\\
& \ \ \ - \sqrt{\tr{c}G(q_x(t)+\mathcal{G}_{\mathcal{E}_{x}(t)}^{-1}(c))R_u(t) G(q_x(t)+\mathcal{G}_{\mathcal{E}_{x}(t)}^{-1}(c)) c}   \\
& \ \ \ + \sqrt{\tr{c} \Omega_{n}(q_x(t),Q_x(t)) c} + \sqrt{\tr{c} B(q_x(t))Q_w B(q_x(t)) c}  \, .
\end{aligned}
\end{equation*}
Using the support function of the ellipsoids $\mathcal E(Q_x(t))$,
\begin{equation*}
V[\mathcal E(Q_x(t))](c) = \sqrt{\tr{c} Q_x(t) c}\,,
\end{equation*}
we thus have that condition \eqref{eq::cond1} is equivalent to 
\begin{equation}
\label{eq::cond2}
\begin{aligned}
&\frac{1}{2}\tr{c} \dot{Q}_x(t) c \geq \tr{c} A(t) Q_{x}(t) c \\
& \quad - \Vert \tr{c} G(q_x(t)+\mathcal{G}_{\mathcal{E}_{x}(t)}^{-1}(c)) R_{u}^\frac{1}{2}(t) \Vert_2 \, \Vert Q_{x}^\frac{1}{2}(t) c \Vert_2\\
& \quad + \Vert Q_{x}^\frac{1}{2}(t)c \Vert_2 \, \Vert \Omega_{n}^\frac{1}{2}(q_x(t),Q_x(t)) c \Vert_2 \\
& \quad + \Vert Q_{x}^\frac{1}{2}(t)c \Vert_2 \, \Vert Q^\frac{1}{2}_{w}\tr{B(q_x(t))} c \Vert_2 \,, 
\end{aligned}
\end{equation}
for a.e. $t\in [t_1,t_2]$ and all $c \in \mathbb R^{n_x}$ with $\tr{c}c=1$.

At this point, we use the following identities,
\begin{align*}
\Vert C_1 y \Vert_2 \Vert C_2 y \Vert_2 = \max_S \; & \tr{y} \tr{C_1}SC_2 y \ \ \text{s.t.}\ \ S\tr{S} \preceq I \nonumber\\
= \inf_{\lambda>0} \; & \frac{1}{2 \lambda} \tr{y}\tr{C_1}C_1y + \frac{\lambda}{2} \tr{y}\tr{C_2}C_2y\,,
\end{align*}
in order to establish that \eqref{eq::cond2} holds whenever there exist real-valued functions $\lambda,\kappa:[t_1,t_2]\to\mathbb R_{++}$ and a matrix-valued function $S:[t_1,t_2]\to\mathbb{R}^{n_x\times n_u}$ with $S\tr{S}\preceq I$ such that
\begin{equation}
\label{eq::cond3}
\begin{aligned}
& \frac{1}{2}\tr{c} \dot{Q}_x(t) c \, \geq  \tr{c} A(t) Q_{x}(t) c \\
&\quad - \tr{c} Q_{x}^\frac{1}{2}(t) S(t) R^\frac{1}{2}_{u}(t) \tr{G(q_x(t)+\mathcal{G}_{\mathcal{E}_{x}(t)}^{-1}(c))} \\
&\quad + \left(\frac{1}{2\lambda(t)}+\frac{1}{2\kappa(t)}\right) \tr{c}Q_{x}(t)c\\
&\quad + \frac{\lambda(t)}{2}\tr{c} \Omega_{n}(q_x(t),Q_x(t))c \\
&\quad + \frac{\kappa(t)}{2}\tr{c} B(q_x(t)) Q_{w} \tr{B(q_x(t))}c \,, 
\end{aligned}
\end{equation}
for a.e. $t\in [t_1,t_2]$ and all $c \in \mathbb R^{n_x}$ with $\tr{c}c=1$. In particular, condition~\eqref{eq::elltube_Qx} along with Assumption~\ref{ass::nonlinearityG} ensure that
\begin{equation*}
\begin{aligned}
\dot{Q}_x&(t) \succeq A(q_x(t))Q_x(t) + Q_x(t)\tr{A(q_x(t))} \\
& + Q_x^\frac{1}{2}(t)S(t)R_{u}^\frac{1}{2}(t)\tr{G(\xi)} + G(\xi)R_{u}^\frac{1}{2}(t)\tr{S(t)}Q_x^\frac{1}{2}(t) \\
& + \left(\frac{1}{\lambda(t)} + \frac{1}{\kappa(t)} \right) Q_x(t) + \lambda(t) B(q_x(t))Q_w\tr{B(q_x(t))} \\
& + \kappa(t)\, \Omega_{n}(q_x(t),Q_x(t)) \, ,
\end{aligned}
\end{equation*}
for a.e. $t\in [t_1,t_2]$ and for all $\xi \in \mathcal E(q_x(t),Q_x(t))$, which also implies condition \eqref{eq::cond3} since $\mathcal{G}_{\mathcal{E}_{x}(t)}^{-1}(c)\in\mathcal E(Q_x(t))$.

The result that $\mathcal E(q_x(t),Q_x(t))$ describes a RFIT on $[t_1,t_2]$ follows from Theorem~\ref{thm::minmaxDI}. Moreover, a feedback control law for this tube is given by $\mu(t,\xi) = \mu^*_t(\mathcal{G}_{\mathcal{E}_{x}(t)}( \xi - q_{x}(t) ))$ with
\begin{equation}
\label{eq::ellmutstar}
\begin{aligned}
\mu^*_t(c) :=\ & \argmin_{\nu\in \mathcal{E}(u_{x}(t),R_u(t))} \tr{c}G\left(q_x(t)+\mathcal{G}^{-1}_{\mathcal{E}_{x}(t)}(c)\right)\nu\\
=\ & u_{x}(t) - \textstyle\frac{R_{u}(t)\tr{G\left(q_x(t)+\mathcal{G}^{-1}_{\mathcal{E}_{x}(t)}(c)\right)}c}{\left\Vert R^\frac{1}{2}_{u}(t)\tr{G\left(q_x(t)+\mathcal{G}^{-1}_{\mathcal{E}_{x}(t)}(c)\right)}c\right\Vert_{2}}\,,
\end{aligned}
\end{equation}
for all $c\neq 0$, i.e. for $\xi\in\bd \mathcal{E}(q_{x}(t),Q_{x}(t))$. Finally, since any control action $u(t)\in\mathcal{E}(q_{u},Q_{u})$ is valid for $\xi$ in the interior of $\mathcal{E}(q_{x}(t),Q_{x}(t))$, and since $u_x(t)$ is the natural control action for $\xi = q_{x}(t)$, we can define a feedback control associated to the ellipsoidal tube by
\begin{equation*}
\mu(t,\xi) = 
\begin{cases}
\mu^{*}_{t}(\mathcal{G}_{Y(t)}(\xi)) & \text{if }\; \xi \in \bd Y(t) \\
u_x(t) &\text{otherwise}
\end{cases}\;.
\end{equation*}

\paragraph*{S2}
In order to show that the result also holds for general positive semidefinite matrices, we can add a small regularization term $\epsilon I$ to the matrices $Q_w$, $Q_u$ and $Q_x(t)$, and then take limits as $\epsilon\to 0$ by invoking the exact same continuity argument as in Step~S2 of the proof of Theorem~\ref{thm::minmaxDI}. This process results in the feedback control law $\mu(t,\xi) = \mu^*_t(\mathcal{G}_{\mathcal{E}_{x}(t)}( \xi - q_{x}(t) ))$, with $\mu^*_t$ given by Eq.~\eqref{eq::ellmutstar}; and the Gauss map $\mathcal{G}_{\mathcal{E}_{x}(t)}$ given by
\begin{equation*}
\mathcal{G}_{\mathcal{E}_{x}(t)}(\xi) = \frac{Q^{\dagger}_{x}(t)\xi}{\left\Vert Q^{\dagger}_{x}(t)\xi \right\Vert_2} \;,
\end{equation*}
which follows from the fact that
\begin{equation*}
\lim_{\epsilon \to 0}  (Q_{x}(t)+\epsilon I)^{-1}\xi = Q^\dagger_{x}(t) \xi
\end{equation*} 
for all $\xi\in\bd\mathcal E(Q_x(t))\subseteq\operatorname{span}(Q_{x}(t))$. \qed

\section{Proof of Lemma~\ref{lem::innercontrol}}
\label{app::ellinnerapprox}

The statement of the Lemma is trivially satisfied with $\gamma(t) = 1$, and this case is thus excluded from the following considerations.
Consider the rank-1 ellipsoid
\begin{equation*}
\mathcal{E}\left(q_u,(u_x(t)-q_u)\tr{(u_x(t)-q_u)}\right)
\end{equation*}
and observe that $\mathcal{E}(u_x(t),R_u(t))\subseteq \mathcal{E}(q_u,Q_u)$ if
\begin{equation}
\label{eq::auxmink1}
\mathcal{E}\left(q_u,(u_x(t)-q_u)\tr{(u_x(t)-q_u)}\right)\oplus \mathcal{E}(R_u(t)) \subseteq \mathcal E(q_u,Q_u)\,.
\end{equation}
Using the
standard formula for the ellipsoidal bounding of the Minkowski sum of ellipsoids~\cite{Kurzhanski1993}, 
we find that~\eqref{eq::auxmink1} holds if
\begin{equation}
\label{eq::auxmink2}
Q_u = \frac{1}{\gamma(t)}(u_x(t)-q_u)\tr{(u_x(t)-q_u)} + \frac{1}{1-\gamma(t)} R_{u}(t) \;,
\end{equation}
for any $\gamma(t) \in (0,1)$. Solving Eq.~\eqref{eq::auxmink2} with respect to $R_u(t)$ yields the statement of Lemma~\ref{lem::innercontrol}.
\qed

\section{Smooth Nonlinearity Bounders for Twice-Continuously-Differentiable Functions}
\label{app::frobeniusbound}

\begin{lemma}
\label{lem::frobeniusbound}
Consider a twice-continuously-differentiable function $g:\mathbb R^{n_y}\to \mathbb R^{n_y}$, and define the remainder function $n:\mathbb{R}^{n_y}\to\mathbb{R}^{n_y}$ such that
\begin{equation*}
g(y) = g(q_y) + \frac{\partial g}{\partial y}(q_y)\delta_y + n(\delta_y)\,, 
\end{equation*}
with $\delta_y := y-q_y$. Let $D_{y}\in\mathbb K^{n_y}$ and $(q_y,Q_y)\in\mathbb{R}^{n_y}\times\mathbb{S}_+^{n_y}$ such that $\mathcal E(q_y,Q_y)\subseteq D_{y}$. Suppose that there exist constants $\bar F_1,\ldots,\bar F_{n_y} \in \mathbb R_{+}$ satisfying
\begin{equation*}
\forall y\in\mathcal D_y\,,\quad \bar F_i \geq \left\Vert \frac{\partial^{2}g_i}{\partial y^{2}}(y)\, S_i\right\Vert_{\rm F}\,, 
\end{equation*}
for certain invertible matrices $S_1,\ldots,S_{n_y}\in\mathbb R^{n_y\times n_y}$. Then, $n(\delta_y) \in \mathcal E\left( Q_{n} \right)$, for all $\delta_y\in\mathcal E(Q_y)$ with
\begin{equation*}
Q_{n} := \frac{1}{4}\operatorname{diag}\left( \bar F_i^{2}\left\Vert S_{i}^{-1}Q_{y}\right\Vert^{2}_{\rm F}\right)_{1\leq i\leq n_y}\,. 
\end{equation*}
\end{lemma}

\begin{proof}
From Taylor's theorem, the remainder function $n_i$ corresponding to $g_i$, for each $i=1,\ldots,n_y$, is given by 
\begin{equation*}
n_i(\delta_y) = \frac{1}{2}\tr{\delta_y}\,\frac{\partial^2g_i}{\partial y^2}(\xi_i)\,\delta_y\,,
\end{equation*}
for some $\xi_i\in\operatorname{conv}(\{ y,q_y \})$. Then, for all $\delta_y\in\mathcal E(Q_y)$, we have
\begin{align*}
n_i(\delta_y)
& = \frac{1}{2}\operatorname{Tr}\left(\frac{\partial^2g_i}{\partial y^2}(\xi_i)\, S_iS_i^{-1}\delta_y\tr{\delta_y}\right)\\
& \leq \frac{1}{2}\operatorname{Tr}\left(\frac{\partial^2g_i}{\partial y^2}(\xi_i)\, S_iS_i^{-1}Q_y\right)\\
& =\frac{1}{2}\left\Vert\frac{\partial^2g_i}{\partial y^2}(\xi_i)\, S_i \left(S_i^{-1}Q_y\right)\right\Vert_{\rm F}\\
& \leq \frac{1}{2}\left\Vert \frac{\partial^2g_i}{\partial y^2}(\xi_i)\,S_i\right\Vert_{\rm F} \left\Vert S_i^{-1} Q_y\right\Vert_{\rm F}\\
& \leq \frac{1}{2}\bar F_i\left\Vert S_i^{-1} Q_y\right\Vert_{\rm F} \,.
\end{align*}
Therefore, $n$ is bounded on $\mathcal E(Q_y)$ by an ellipsoid centered at the origin and with semi-axes of length $\frac{1}{2}\bar F_i\Vert S_i^{-1} Q_y \Vert_{\rm F}$.
\qed
\end{proof}

\end{document}